\documentclass{amsart}
\usepackage{graphicx}
\usepackage{calc}

\usepackage[cmtip,arrow,all,2cell]{xy}
\usepackage{pb-diagram,pb-xy}
\usepackage{amsmath,amsthm,amssymb,tabularx}

\newcommand{\norm}[1]{\left\Vert#1\right\Vert}
\newcommand{\abs}[1]{\left\vert#1\right\vert}

\theoremstyle{plain}

\newtheorem{theorem}{Theorem}

\newcommand{\beq}{\begin{equation}}
\newcommand{\eeq}{\end{equation}}

\def\C{\mathbb{C}}
\def\N{\mathbb{N}}
\def\e{\varepsilon}

\def\Supp{\operatorname{\sf{supp}}}

\begin{document}

\title[Stereotype approximation for the algebra ${\mathcal C}(M)$]{Stereotype approximation property for the algebras ${\mathcal C}(M)$ of continuous functions on metric spaces}

\author{S.S.Akbarov}
 \address{School of Applied Mathematics, National Research University Higher School of Economics, 34, Tallinskaya St. Moscow, 123458 Russia}
 \email{sergei.akbarov@gmail.com}

\thanks{Supported by the RFBR grant No. 18-01-00398.}

\maketitle

\begin{abstract}
In this note we prove that the algebra ${\mathcal C}(M)$ of continuous functions on an arbitrary complete (not necessarily locally compact) metric space $M$ has the stereotype approximation property.
\end{abstract}

In \cite{Akbarov-approx,Akbarov-2003} the author described the stereotype approximation property, an analog of the classical approximation property transferred into the category ${\tt Ste}$ of stereotype spaces.
Like the classical approximation, the stereotype approximation is used in the description of properties of spaces of operators (and in this way the stereotype approximation possesses some advantages, since in contrast to the classical approximation the stereotype approximation is inherited by spaces of operators \cite[Theorem 9.9]{Akbarov-2003}). This justifies the study of this notion, but the question of which concrete spaces in the standard package used in functional analysis have the stereotype approximation turns out to be much more difficult than in the classical theory, since the spaces of operators in the category ${\tt Ste}$ are defined in a more complicated way.

For this reason, each result on the stereotype approximation of a particular space is of interest, as it requires the development of new techniques.
In line of these studies the author showed in \cite{Akbarov-approximation-C(G)} that the group algebra $\mathcal C^\star(G)$ of measures on an arbitrary locally compact group $G$, as well as its dual algebra ${\mathcal C}(G)$ of continuous functions on $G$, have the stereotype approximation. Whether the same is true for the algebra ${\mathcal C}(M)$ of continuous functions on an arbitrary paracompact locally compact space $M$, even in the case of compact $M$, remains an open problem. In this note, we give an answer to this question for the case of an arbitrary complete (not necessarily locally compact) metric space $M$.

Let $M$ be a complete metric space with the distance $d$. Let us denote by ${\mathcal C}_\natural(M)$ the algebra of continuous functions $f:M\to\C$ with the usual pointwise multiplication and the topology of uniform convergence on compact sets in $M$. In other words, the topology on ${\mathcal C}_\natural(M)$ is defined by the system of seminorms
\beq\label{norm(f)_T=max_t-abs(f(t))}
\norm{f}_T=\max_{t\in T}\abs{f(t)},
\eeq
where $T$ runs over the set of all compact sets in $M$. The multiplication operation in ${\mathcal C}_\natural(M)$ is a continuous bilinear form in the sense of the theory of stereotype spaces: for each neighbourhood of zero $U\subseteq {\mathcal C}_\natural(M)$ and for each totally bounded set $F\subseteq {\mathcal C}_\natural(M)$ there is a neighbourhood of zero $V\subseteq {\mathcal C}_\natural(M)$ such that $V\cdot F=F\cdot V\subseteq U.$ By \cite[Theorem 5.23]{Akbarov-2003} this implies that under the pseudosaturation $\vartriangle$ \cite[$\S$ 1.4]{Akbarov-2003} the multiplication operation remains a continuous bilinear form in the same sense. As a corollary, the pseudosaturation ${\mathcal C}_\natural(M)^\vartriangle$ of the space ${\mathcal C}_\natural(M)$ is a stereotype algebra. Below by ${\mathcal C}(M)$ we denote this algebra ${\mathcal C}_\natural(M)^\vartriangle$:
$$
{\mathcal C}(M):={\mathcal C}_\natural(M)^\vartriangle.
$$

\begin{theorem}\label{TH:C(M)-obl-approx-kogda-M-metricheskoe}
For each complete metric space $M$ the algebra ${\mathcal C}(M)$ has the stereotype approximation property.
\end{theorem}
\begin{proof}
1. For each $k\in\N$ consider the system $\{U_{\frac{1}{k}}(t); t\in M\}$ of open balls of radius $\frac{1}{k}$ in $M$. This is an open covering of $M$, hence it has a subordinated locally finite partition of unity $\{\eta^k_t;t\in M\}$ \cite[Theorem 5.1.3, Theorem 5.1.9]{Engelking}:
$$
\eta^k_t\in{\mathcal C}(M),\quad 0\le \eta^k_t\le 1,\quad \Supp\eta^k_t\subseteq U_{\frac{1}{k}}(t),\quad \sum_{t\in M}\eta^k_t=1.
$$
Let $M^k\subseteq M$ be the set of points $t\in M$ for which the function $\eta^k_t$ does not vanish: $t\in M^k$ $\Leftrightarrow$ $\eta^k_t\ne 0.$

2. Let us consider the sequence of operators $P^k:{\mathcal C}(M)\to {\mathcal C}(M)$ defined by the equality
\beq\label{C(M)-obl-approx-kogda-M-metricheskoe-1}
P^kf=\sum_{t\in M^k} f(t)\cdot \eta^k_t,\qquad f\in{\mathcal C}(M).
\eeq
Since the series on the right is locally finite (i.e. in a neighbourhood of each point  $x\in M$ only a finite set of its terms do not vanish), the function $P^kf\in{\mathcal C}(M)$ is well defined, and we obtain a linear map $P^k:{\mathcal C}(M)\to {\mathcal C}(M)$.

Note that this map is continuous in the topology of ${\mathcal C}_\natural(M)$. Indeed, if $\{f_i;i\to\infty\}$ is a net, tending to zero in ${\mathcal C}_\natural(M)$, i.e. uniformly on compact sets $T\subseteq M$, $\max_{x\in T}\abs{f_i(x)}\underset{i\to\infty}{\longrightarrow}0,$ then, since the set
\beq\label{M^k_T=(t-in-T:eta^k_t|_T-ne-0)}
M^k_T=\{t\in T: \eta^k_t\Big|_T\ne 0\}
\eeq
is finite, we have
\begin{multline*}
\max_{x\in T}\abs{P^kf_i(x)}=
\max_{x\in T}\abs{\sum_{t\in M^k} f_i(t)\cdot \eta^k_t(x)}=
\max_{x\in T}\abs{\sum_{t\in M^k_T} f_i(t)\cdot \eta^k_t(x)}\le\\ \le
\max_{x\in T}\sum_{t\in M^k_T}\abs{f_i(t)\cdot \eta^k_t(x)}\le
\max_{x\in T}\max_{t\in M^k_T}\abs{f_i(t)}\cdot \underbrace{\sum_{t\in M^k_T}\eta^k_t(x)}_{\scriptsize\begin{matrix}\| \\1 \end{matrix}}\le
\max_{t\in M^k_T} \abs{f_i(t)}
\underset{i\to\infty}{\longrightarrow}0.
\end{multline*}

Since the operators $P^k:{\mathcal C}_\natural(M)\to {\mathcal C}_\natural(M)$ are continuous, their pseudosaturations $(P^k)^\vartriangle:{\mathcal C}(M)^\vartriangle\to {\mathcal C}(M)^\vartriangle$, i.e. the operators $P^k:{\mathcal C}(M)\to {\mathcal C}(M)$, are continuous as well \cite[Theorem 1.16]{Akbarov-2003}.

3. Let us show further that the representation \eqref{C(M)-obl-approx-kogda-M-metricheskoe-1} of the operator $P^k$ can be understood as an expantion of $P^k$ into a converging series of one-dimensional operators in  ${\mathcal L}({\mathcal C}(M))$:
\beq\label{C(M)-obl-approx-kogda-M-metricheskoe-2}
P^k=\sum_{t\in M^k} \delta_t\odot \eta^k_t.
\eeq
($\delta_t$ is the $\delta$-functional at the point $t$). This is done in several steps. Let us first show that   \eqref{C(M)-obl-approx-kogda-M-metricheskoe-2} can be understood as the convergence of a series in the space ${\mathcal C}_\natural(M):{\mathcal C}(M)$ (we use the notations of \cite[$\S 5$]{Akbarov-2003}: according to them, ${\mathcal C}_\natural(M):{\mathcal C}(M)$ is the space of linear continuous maps $\varphi:{\mathcal C}(M)\to{\mathcal C}_\natural(M)$ with the topology of uniform convergence on compact sets). For this let us take a totally bounded set $F\subseteq {\mathcal C}(M)$ and a basis neighbourhood of zero in ${\mathcal C}_\natural(M)$, i.e. a set of the form $B_{\e}(T)=\{g\in {\mathcal C}_\natural(M): \norm{g}_T<\e\},$ where $T$ is a compact set in $M$. Recall the set $M^k_T$ defined in \eqref{M^k_T=(t-in-T:eta^k_t|_T-ne-0)}. It is finite (since the family $\eta^k_t$ is locally finite), hence for each finite set $N\supseteq M^k_T$ we have
\begin{multline*}
\sum_{t\in N} (\delta_t\odot\eta^k_t)(f)\Big|_T=
\sum_{t\in N} f(t)\cdot \eta^k_t\Big|_T=\sum_{t\in N\setminus M^k_T} f(t)\cdot \underbrace{\eta^k_t\Big|_T}_{\scriptsize\begin{matrix}\|\\ 0\end{matrix}}+\sum_{t\in M^k_T} f(t)\cdot \eta^k_t\Big|_T=\\=
\sum_{t\in M^k_T} f(t)\cdot \eta^k_t\Big|_T=P^kf\Big|_T
\end{multline*}
This implies the following chain of corollaries:
$\Big(\forall N\supseteq M^k_T\quad \forall f\in F$ $\|\sum_{t\in N} (\delta_t\odot\eta^k_t)(f)-P^kf\|_T=0\Big)$
$\Longrightarrow$
$\Big(
\forall N\supseteq M^k_T\quad \forall f\in F\quad \sum_{t\in N} (\delta_t\odot\eta^k_t)(f)-P^kf\in B_{\e}(T)\Big)
$
$\Longrightarrow$
$\Big(\forall N\supseteq M^k_T\quad \sum_{t\in N}\delta_t\odot\eta^k_t-P^k\in B_{\e}(T):F\Big)$ (here $B_{\e}(T):F$ denotes the set of operators $\varphi:{\mathcal C}(M)\to {\mathcal C}_\natural(M)$ with the property $\varphi(B_{\e}(T))\subseteq F$, see \cite[5.4]{Akbarov-2003}; by definition of the topology in ${\mathcal C}_\natural(M):{\mathcal C}(M)$, this is a basic neighbourhood of zero in  ${\mathcal C}_\natural(M):{\mathcal C}(M)$). And this proves that \eqref{C(M)-obl-approx-kogda-M-metricheskoe-2} indeed can be treated as an equality in the space ${\mathcal C}_\natural(M):{\mathcal C}(M)$.

Further, the convergence of the series in \eqref{C(M)-obl-approx-kogda-M-metricheskoe-2} in ${\mathcal C}_\natural(M):{\mathcal C}(M)$ implies that its partial sums
$P^k_N=\sum_{t\in N} \delta_t\odot \eta^k_t,$ $N\in 2_{M^k}$
($2_{M^k}$ means the set of finite subsets in $M^k$) form a totally bounded set in ${\mathcal C}_\natural(M):{\mathcal C}(M)$ \cite[Proposition 9.18]{Akbarov-2003}. By \cite[Theorem 5.1]{Akbarov-2003}, this means that the set of operators $\{P^k_N;N\in 2_{M^k}\}$ is equicontinuous on each totally bounded set $F\subseteq {\mathcal C}(M)$ and has a totally bounded image on it $\bigcup_{N\in 2_{M^k}}P^k_N(F)\subseteq {\mathcal C}_\natural(M)$. Since the system of totally bounded sets in ${\mathcal C}_\natural(M)$ and in ${\mathcal C}(M)$ is the same, and the topology on them does not change either \cite[Theorem 1.17]{Akbarov-2003}, we can conclude that $\bigcup_{N\in 2_{M^k}}P^k_N(F)$ is totally bounded in ${\mathcal C}(M)$, and $\{P^k_N;N\in 2_{M^k}\}$ is equicontinuous as a system of mappings from $F$ (with the uniform structure induced from ${\mathcal C}(M)$) into  $\bigcup_{N\in 2_{M^k}}P^k_N(F)\subseteq {\mathcal C}(M)$, or, in other words, as a system of mappings from $F$ into ${\mathcal C}(M)$. Since this is true for each totally bounded set $F\subseteq {\mathcal C}(M)$, by \cite[Theorem 5.1]{Akbarov-2003} we conclude that the system of operators $\{P^k_N;N\in 2_{M^k}\}$ is a totally bounded set in ${\mathcal C}(M):{\mathcal C}(M)$ (with the removed symbol $\natural$ in the numerator).

The closure $\overline{\{P^k_N;N\in 2_{M^k}\}}$ of this set in ${\mathcal C}(M):{\mathcal C}(M)$ is compact in ${\mathcal C}(M):{\mathcal C}(M)$. The topology of ${\mathcal C}_\natural(M):{\mathcal C}(M)$, being formally coarser than the topology of ${\mathcal C}(M):{\mathcal C}(M)$, is nevertheless Hausdorff, and thus, it separates the points of $\overline{\{P^k_N;N\in 2_{M^k}\}}$. As a corollary, these topologies coincide on the set $\overline{\{P^k_N;N\in 2_{M^k}\}}$. Therefore, the convergence of the net $\{P^k_N;N\in 2_{M^k}\}$ to the element $P^k\in \overline{\{P^k_N;N\in 2_{M^k}\}}$ in the topology of ${\mathcal C}_\natural(M):{\mathcal C}(M)$ implies the convergence of $\{P^k_N;N\in 2_{M^k}\}$ to $P^k\in \overline{\{P^k_N;N\in 2_{M^k}\}}$ in the topology of ${\mathcal C}(M):{\mathcal C}(M)$: $P^k_N\overset{{\mathcal C}(M):{\mathcal C}(M)}{\underset{N\to M}{\longrightarrow}}P^k.$

Since the set $\{P^k_N;N\in 2_{M^k}\}\cup\{P^k\}$ is totally bounded in ${\mathcal C}(M):{\mathcal C}(M)$, the passage to the pseudosaturation $\vartriangle$ does not change its topology. We can conclude that the net
$\{P^k_N;N\in 2_{M^k}\}$ converges to $P^k$ in the topology of $({\mathcal C}(M):{\mathcal C}(M))^\vartriangle={\mathcal L}({\mathcal C}(M))$:
$P^k_N\overset{{\mathcal L}({\mathcal C}(M))}{\underset{N\to M}{\longrightarrow}}P^k.$ This means that \eqref{C(M)-obl-approx-kogda-M-metricheskoe-2} holds in the space ${\mathcal L}({\mathcal C}(M))$.

4. Let us show that the operators $P^k$ approximate the identity operator $I$ in the space ${\mathcal L}({\mathcal C}(M))$:
\beq\label{C(M)-obl-approx-kogda-M-metricheskoe-5}
P^k\overset{{\mathcal L}({\mathcal C}(M))}{\underset{N\to M}{\longrightarrow}}I.
\eeq
This is also done in several steps. Again, let us take a totally bounded set $F\subseteq {\mathcal C}(M)$ and a compact set $T\subseteq M$. Then
\begin{multline*}
  \norm{P^kf-f}_T=\max_{x\in T}\abs{P^kf(x)-f(x)}=
  \max_{x\in T}\Big|\sum_{t\in M^k} f(t)\cdot \eta^k_t(x)-f(x)\Big|=\\=
  \max_{x\in T}\Big|\sum_{t\in M^k} f(t)\cdot \eta^k_t(x)-f(x)\cdot \underbrace{\sum_{t\in M^k}\eta^k_t(x)}_{\scriptsize\begin{matrix}\|\\ 1\end{matrix}}\Big|=
  \max_{x\in T}\Big|\sum_{t\in M^k} \big(f(t)-f(x)\big)\cdot \eta^k_t(x)\Big|\le \\ \le
  \max_{x\in T}\sum_{t\in M^k} \Big|f(t)-f(x)\Big|\cdot \eta^k_t(x)
  \end{multline*}
In the last sum if $\eta^k_t(x)\ne 0$, then $x\in \Supp\eta^k_t\subseteq U_{\frac{1}{k}}(t)$, hence $d(x,t)<\frac{1}{k}$. Thus,
\begin{multline}\label{C(M)-obl-approx-kogda-M-metricheskoe-3}
\norm{P^kf-f}_T\le \max_{x,s\in T}\sum_{t\in M^k} \kern-22pt \underbrace{\Big|f(t)-f(x)\Big|}_{ \scriptsize\begin{matrix}\text{\rotatebox{90}{$\ge$}} \\
\sup_{x,s\in T,\ d(s,x)< \frac{1}{k}}\Big|f(s)-f(x)\Big| \end{matrix}}\kern-22pt \cdot \eta^k_t(x)\le\\ \le
\sup_{x,s\in T,\ d(s,x)< \frac{1}{k}}\Big|f(s)-f(x)\Big|\cdot \underbrace{\sum_{t\in M^k}
\eta^k_t(x)}_{\scriptsize\begin{matrix}\|\\ 1\end{matrix}}=
\sup_{x,s\in T,\ d(s,x)< \frac{1}{k}}\Big|f(s)-f(x)\Big|
\end{multline}
Let us show that the last value tends to zero uniformly by $f\in F$:
\beq\label{C(M)-obl-approx-kogda-M-metricheskoe-4}
\sup_{f\in F}\sup_{x,s\in T,\ d(s,x)< \frac{1}{k}}\Big|f(s)-f(x)\Big|\underset{k\to\infty}{\longrightarrow}0
\eeq
Suppose that this is not true: $\sup_{f\in F}\sup_{x,s\in T,\ d(s,x)< \frac{1}{k}}\Big|f(s)-f(x)\Big|\underset{k\to\infty}{\not\longrightarrow}0$.
This means that there is a number $\e>0$ and sequences $k_n\to \infty$, $f_n\in F$, $x_n\in T$, $s_n\in U_{\frac{1}{k_n}}(x_n)$ such that $\forall n\in\N\quad \Big|f_n(s_n)-f_n(x_n)\Big|>\e.$ Since $T$ is compact, we can choose from $x_n\in T$ a converging subsequence $x_{n_i}$:
$x_{n_i}\underset{i\to\infty}{\longrightarrow}x\in T.$
Then $s_{n_i}\underset{i\to\infty}{\longrightarrow}x\in T,$
and we have the inequality $\forall i\in\N\quad \Big|f_{n_i}(s_{n_i})-f_{n_i}(x_{n_i})\Big|>\e,$
where $f_{n_i}\in F$, $x_{n_i}\to x$ and $s_{n_i}\to x$. This means that the set of functions $F$ is not equicontinuous of the compact set $\{x_{n_i}\}\cup \{s_{n_i}\}\cup\{x\}$, and therefore $F$ is not totally bounded in ${\mathcal C}(M)$ \cite[8.2.10]{Engelking}. And this contradicts to the choice of $F\subseteq {\mathcal C}(M)$.

We proved \eqref{C(M)-obl-approx-kogda-M-metricheskoe-4} and together with \eqref{C(M)-obl-approx-kogda-M-metricheskoe-3} this gives
$0\le\sup_{f\in F}\norm{P^kf-f}_T\le \sup_{f\in F}\sup_{x\in T,\ d(s,x)< \frac{1}{k}}\Big|f(s)-f(x)\Big|\underset{k\to\infty}{\longrightarrow}0,$
and therefore, $\sup_{f\in F}\norm{P^kf-f}_T\underset{k\to\infty}{\longrightarrow}0.$
This is true for each compact set $T\subseteq M$, so we can say that for each totally bounded set $F\subseteq {\mathcal C}(M)$ the net $P^kf-f$ tends to zero in the space ${\mathcal C}_\natural(M)$ uniformly by $f\in F$:
\beq\label{C(M)-obl-approx-kogda-M-metricheskoe-6}
P^kf-f\overset{{\mathcal C}_\natural(M)}{\underset{f\in F}{\underset{k\to\infty}{\rightrightarrows}}}0.
\eeq
Recall now that $P^k$ is not just a net, but a {\it sequence}. Together with \eqref{C(M)-obl-approx-kogda-M-metricheskoe-6} this gives that the set $\bigcup_{k\in \N}(P^k-I)(F)$ must be totally bounded in ${\mathcal C}_\natural(M)$. Hence the pseudosaturation $\vartriangle$ does not change the topology on $\bigcup_{k\in \N}(P^k-I)(F)\cup\{0\}$ \cite[Theorem 1.17]{Akbarov-2003}, and we can conclude that
$P^kf-f$ tends to zero in the space ${\mathcal C}_\natural(M)^\vartriangle={\mathcal C}(M)$ uniformly by $f\in F$:
\beq\label{C(M)-obl-approx-kogda-M-metricheskoe-7}
P^kf-f\overset{{\mathcal C}(M)}{\underset{f\in F}{\underset{k\to\infty}{\rightrightarrows}}}0.
\eeq
And this is true for each totally bounded set $F\subseteq {\mathcal C}(M)$. Hence, $P^k\overset{{\mathcal C}(M):{\mathcal C}(M)}{\underset{N\to M}{\longrightarrow}}I.$ Recall again that $P^k$ is a sequence. It converges to $I$ in the space ${\mathcal C}(M):{\mathcal C}(M)$, therefore the set $\{P^k\}\cup\{I\}$ is compact. As a corollary, $P^k$ tends to $I$ in the topology of the compact set $\{P^k\}\cup\{I\}\subseteq {\mathcal C}(M):{\mathcal C}(M)$. When we apply to the space ${\mathcal C}(M):{\mathcal C}(M)$ the operation of pseudosaturation $\vartriangle$, the topology on the set $\{P^k\}\cup\{I\}$ is not changed \cite[Theorem 1.17]{Akbarov-2003}. Hence we can say that $P^k$ tends to $I$ in the topology of the compact set $\{P^k\}\cup\{I\}\subseteq ({\mathcal C}(M):{\mathcal C}(M))^\vartriangle={\mathcal L}({\mathcal C}(M))$. Thus, $P^k$ tends to $I$ in the space ${\mathcal L}({\mathcal C}(M))$. In other words, \eqref{C(M)-obl-approx-kogda-M-metricheskoe-5} holds.

5. We see that the identity operator $I$ is approximated in the space ${\mathcal L}({\mathcal C}(M))$ by the operators $P^k$ (by \eqref{C(M)-obl-approx-kogda-M-metricheskoe-5}), and the operators $P^k$ are approximated in ${\mathcal L}({\mathcal C}(M))$ by finite-dimensional operators (the partial sums of the series in \eqref{C(M)-obl-approx-kogda-M-metricheskoe-2}). Thus, $I$ is approximated in ${\mathcal L}({\mathcal C}(M))$ by finite-dimensional operators, and this is the stereotype approximation property for ${\mathcal C}(M)$.
\end{proof}

\end{document}